\documentclass[reqno]{amsart}
\usepackage{amssymb}
\usepackage{amsthm}
\usepackage{mathtools}
\mathtoolsset{showonlyrefs, showmanualtags}

\usepackage{aligned-overset}

\usepackage{comment}

\usepackage{mleftright}
\usepackage{enumitem}
\setlist[enumerate]{
	leftmargin=*,
	align=left
}

\numberwithin{equation}{section}

\theoremstyle{plain}
\newtheorem{theorem}{Theorem}[section]
\newtheorem{corollary}[theorem]{Corollary}
\newtheorem{proposition}[theorem]{Proposition}
\newtheorem{conjecture}{Conjecture}

\newcommand{\R}{\mathbb{R}}
\newcommand{\N}{\mathbb{N}}
\newcommand{\C}{\mathbb{C}}
\renewcommand{\S}{\mathbb{S}}

\DeclareMathOperator{\Real}{Re}
\renewcommand{\Re}{\Real}

\DeclareMathOperator{\dist}{dist}
\NewDocumentCommand{\distfX}{O{f} O{X}}{\dist(#1, #2)}

\DeclareMathOperator{\Span}{span}

\newcommand{\Laplacian}{\Delta}

\newcommand{\factorial}[1]{#1!}

\newcommand{\abs}[1]{\lvert#1\rvert} 
\newcommand{\Abs}[1]{\mleft \lvert #1 \mright \rvert}
\newcommand{\norm}[1]{\lVert#1\rVert} 
 
\newcommand{\innerproduct}[2]{\langle #1, #2 \rangle}
\newcommand{\set}[2]{\{ \, #1 : #2 \, \} } 
\newcommand{\Set}[2]{ \mleft \{ \, #1 : #2 \, \mright\} }

\newcommand{\lambdak}[1][k]{\lambda_{#1}}

\newcommand{\intervaloo}[2]{(#1, #2)}
\newcommand{\intervalco}[2]{[#1, #2)}

\newcommand{\intervalcc}[2]{[#1, #2]}

\newcommand{\SO}{\textup{SO}}
\newcommand{\Rotate}{R}
\newcommand{\RotateV}[1][\planeV]{\Rotate_{#1}}
\newcommand{\RotateVperp}[1][\planeVperp]{\RotateV[\planeVperp]}
\newcommand{\Rotatemean}{\mathcal{R}}

\NewDocumentCommand{\Vst}{O{\sigma} O{\tau}}{\planeV_{#1, #2}}
\NewDocumentCommand{\ProjVst}{O{\sigma} O{\tau}}{\Rotatemean_{#1, #2}}
\NewDocumentCommand{\omegast}{O{\sigma} O{\tau}}{\omega_{#1, #2}}
\NewDocumentCommand{\Rotatest}{O{\sigma} O{\tau}}{\Rotate_{#1, #2}}

\newcommand{\rad}{\textup{rad}}
\newcommand{\even}{\textup{even}}

\newcommand{\Lrad}{L^2_{\rad}}
\newcommand{\Leven}{L^2_{\even}}
\newcommand{\Lk}[1][k]{L^2_{#1}}
\newcommand{\LV}[1][V]{L^2_{#1}}

\newcommand{\Hk}[1][k]{\mathcal{H}_{#1}}
\newcommand{\Hkeven}[1][k]{\mathcal{H}_{#1, \even}}

\newcommand{\Gaussians}{G}

\newcommand{\planeV}{V}
\newcommand{\planeVperp}{\planeV^\perp}
\newcommand{\ProjV}[1][\planeV]{\Rotatemean_{#1}}
\newcommand{\ProjVperp}{\ProjV[\planeVperp]}

\newcommand{\Projrad}{\Pi_{\rad}}
\newcommand{\Projeven}{P_{\even}}

\newcommand{\ProjHk}[1][k]{\Pi_{#1}}

\newcommand{\Projmean}{\mathcal{T}}

\newcommand{\LegendreP}[1][k]{\JacobiP[#1][0][0]}
\NewDocumentCommand{\GegenbauerC}{O{k} O{\mu}}{C_{#1}^{(#2)}}
\NewDocumentCommand{\JacobiP}{O{k} O{\alpha} O{\beta}}{P_{#1}^{(#2, #3)}}

\newcommand{\Zk}[1][k]{Z_{#1}}
\newcommand{\Wk}[1][k]{W_{#1}}
\newcommand{\Zktilde}[1][k]{\widetilde{Z}_{#1}}

\NewDocumentCommand{\Ykmn}{O{k} O{m} O{n}}{Y_{#1}^{(#2, #3)}}
\NewDocumentCommand{\Ykmnab}{O{k} O{m} O{n} O{a} O{b}}{Y_{#1; #4, #5}^{(#2, #3)}}
\NewDocumentCommand{\Ykmnabtilde}{O{4k} O{2m} O{2n} O{0} O{0}}{\widetilde{Y}_{#1; #4, #5}^{(#2, #3)}}
\NewDocumentCommand{\normalizeN}{O{k} O{m} O{n}}{N_{#1}^{(#2, #3)}}

\NewDocumentCommand{\Yjk}{O{j} O{k}}{Y_{#2}^{(#1)}}

\newcommand{\unitvec}{\mathbf{e}}
\newcommand{\unitvecj}[1][j]{\unitvec_{#1}}
\newcommand{\basisE}{E}
\newcommand{\orthogonalE}{E}
\NewDocumentCommand{\orthogonalEst}{O{\sigma} O{\tau}}{\orthogonalE_{#1, #2}}

\usepackage[numbers,sort]{natbib}

\usepackage{hyperref}
\hypersetup{
	setpagesize=false,
	bookmarksnumbered=true,
	bookmarksopen=true,
	colorlinks=true,
	linkcolor=blue,
	citecolor=blue,
}

\usepackage[T1]{fontenc}
\usepackage{lmodern}

\title[A sharpened Strichartz estimate]{A sharpened $L^8_t L^4_x$-Strichartz estimate for even initial data}
\date{}

\author{Soichiro Suzuki}
\address[Soichiro Suzuki]{Department of Mathematics, Chuo University, 1-13-27 Kasuga, Bunkyo-ku, Tokyo, 112-8551, Japan}
\email{soichiro.suzuki.m18020a@gmail.com}



\begin{document}
\begin{abstract}
It is well known that $u_0 \in L^2(\R) \setminus \{0\}$ is an extremizer for the $L^8_t L^4_x$-Strichartz estimate if and only if it is a Gaussian. 
We show that every extremizing sequence of even functions $\{ u_n \} \subset L^2(\R) \setminus \{0\}$ approaches a centered Gaussian in an appropriate sense.  
This result follows from a certain inequality for coordinate-wise even $L^2$ functions on $\R^4$ conjectured by \citeauthor{Gon2019} (\citeyear{Gon2019}) which we prove in this paper with the optimal constant.
\end{abstract}
\maketitle

\section{Introduction} \label{section:introduction}
The Strichartz estimate states that there exists a constant $C > 0$ such that
\begin{equation} \label{eq:Strichartz} \noeqref{eq:Strichartz}
\norm{ \norm{ e^{i t \Laplacian} u_0(x) }_{L^p(\R^d_x)} }_{L^q(\R_t)} \leq C \norm{u_0}_{L^2(\R^d)}
\end{equation}
holds for every $u_0 \in L^2(\R^d)$, where $(p, q, d)$ are such that 
\begin{equation} \label{eq:admissible} \noeqref{eq:admissible}
p, q \in \intervalcc{2}{\infty}, \quad (p, q, d) \ne (\infty, 2, 2), \quad \frac{d}{p} + \frac{2}{q} = \frac{d}{2} .
\end{equation}
This was originally proved by \citet[Corollary 1]{Str1977} for the diagonal case $p = q = 2(d+2)/d$. For the mixed-norm estimates, see \citet[Lemma 3.1 and its proof]{GV1985}, \citet[Lemma 2.1]{Yaj1987}, and \citet[Corollary 1.4]{KT1998}. 
In particular, we have
\begin{align}
\norm{ e^{i t \Laplacian} u_0(x) }_{L^4(\R^2_x \times \R_t)}^4 &\leq \frac{1}{4} \norm{u_0}_{L^2(\R^2)}^4 ,
\label{eq:442}\\
\norm{ e^{i t \Laplacian} u_0(x) }_{L^6(\R_x \times \R_t)}^6 &\leq \frac{1}{2 \sqrt{3}} \norm{u_0}_{L^2(\R)}^6 ,
\label{eq:661}\\
\norm{ \norm{ e^{i t \Laplacian} u_0(x) }_{L^4(\R_x)} }_{L^8(\R_t)}^8 &\leq \frac{1}{4} \norm{u_0}_{L^2(\R)}^8
\label{eq:481}
\end{align}
when $(p, q, d) = (4, 4, 2), (6, 6, 1), (4, 8, 1)$. The constants $1/4$ and $1/(2 \sqrt{3})$ are optimal here, and $u_0 \in L^2(\R^d) \setminus \{0\}$ is an extremizer of these inequalities if and only if
\begin{equation} \label{eq:Gaussian}
u_0(x) = C \exp( A\abs{x}^2 + b \cdot x)
\end{equation}
with $A, C \in \C \setminus \{0\}$, $b \in \C^d$, and $\Re A < 0$. See \citet[Theorems 1.1 and 1.4]{Fos2007}, \citet[Theorems 1.3 and 1.5]{HZ2006}, \citet[Corollary 2]{Car2009}, \citet[Theorem 1.1]{BBCH2009}, and the references therein. 
Now let $\Lrad(\R^d)$ and $\Leven(\R^d)$ be the subspaces of $L^2(\R^d)$ consisting of all radially symmetric functions and of all \emph{coordinate-wise} even functions. In addition, we write
\begin{equation} \label{eq:dist}
\distfX \coloneqq \inf_{ g \in X } \norm{ f - g }_{L^2(\R^d)} 
\end{equation}
for each $f \in L^2(\R^d)$ and $X \subset L^2(\R^d)$.
\citet{Gon2019} proved the following result, which improves the inequality \eqref{eq:442} for $u_0 \in \Lrad(\R^2)$.
\begin{theorem}[{\cite[Theorem 3]{Gon2019}}] \label{thm:442}
Let $f \in L^2(\R^4)$ be such that
\begin{equation}
f(x) = g( \sqrt{x_1^2 + x_2^2}, \sqrt{x_3^2 + x_4^2} )
\end{equation}
for some $g \colon \intervalco{0}{\infty}^2 \to \C$, and let $U \colon \R_x^2 \times \R_t \to \C$ be
\begin{equation}
U(x_1, x_2, t) \coloneqq (e^{i t \Laplacian} f) (x_1, x_2, x_1, x_2).
\end{equation}
Then we have
\begin{equation} \label{eq:442 improved}
	\norm{U}_{L^2(\R^2_x \times \R_t)}^2
	\leq \frac{1}{4} \mleft( \norm{f}_{L^2(\R^4)}^2
	- \gamma \distfX[f][\Lrad(\R^4)]^2 \mright)
\end{equation}
for some constant $\gamma > 0$ independent of $f$.
\end{theorem}
Observe that the inequality \eqref{eq:442} can be improved for $u_0 \in \Lrad(\R^2)$ by letting
\begin{equation}
f(x) = (u_0 \otimes u_0)(x) = u_0(x_1, x_2) u_0(x_3, x_4) .
\end{equation}
\citet{Gon2019} also conjectured the following, which imply refinements of the inequalities \eqref{eq:661} and \eqref{eq:481} for $u_0 \in \Leven(\R)$, respectively.
\begin{conjecture}[{\cite[Conjecture 1]{Gon2019}}] \label{conj:661}
Let $f \in \Leven(\R^3)$,
and let 
\begin{equation}
U(x, t) \coloneqq (e^{i t \Laplacian} f) (x, x, x) .
\end{equation}
Then we have
\begin{equation} \label{eq:661 improved}
	\norm{U}_{L^2(\R_x \times \R_t)}^2
	\leq \frac{1}{2 \sqrt{3}} \mleft( \norm{f}_{L^2(\R^3)}^2
	- \alpha \distfX[f][\Lrad(\R^3)]^2 \mright)
\end{equation}
for some constant $\alpha > 0$ independent of $f$.
\end{conjecture}
\begin{conjecture}[{\cite[Conjecture 2]{Gon2019}}] \label{conj:481}
Let $f \in \Leven(\R^4)$, and let 
\begin{equation}
	U(x, t) \coloneqq (e^{i t \Laplacian} f ) (x_1, x_2, x_1, x_2) .
\end{equation}
Then we have
\begin{equation} \label{eq:481 improved}
		\norm{U}_{L^2(\R^2_x \times \R_t)}^2
	\leq \frac{1}{4} \mleft( \norm{f}_{L^2(\R^4)}^2
	- \beta \distfX[f][\Lrad(\R^4)]^2 \mright)
\end{equation}
for some constant $\beta > 0$ independent of $f$.
\end{conjecture}
\citet[Corollary 1.2]{GZ2022} proved Conjecture \ref{conj:661} with the optimal constant
\begin{equation}
\alpha =  1 - \frac{8320}{3^9} = \frac{11363}{19683} \approx 0.57730 
\end{equation}
and obtained a refinement of the inequality \eqref{eq:661} for $u_0 \in \Leven(\R)$ by letting
\begin{equation}
	f(x) = (u_0 \otimes u_0 \otimes u_0)(x)  .
\end{equation}
Moreover, as claimed in \cite[page 1712]{GZ2022}, their proof also allows us to find the optimal constant $\gamma = 3/4$ in Theorem \ref{thm:442}.
Our aim in this paper is to prove the remaining Conjecture \ref{conj:481} with the optimal constant.
\begin{theorem} \label{thm:481}
Conjecture \ref{conj:481} holds with the constant
\begin{equation} \label{eq:sharp gamma}
\beta = \frac{17 - \sqrt{97}}{32} \approx 0.22347 ,
\end{equation}
which is the largest possible.
\end{theorem}
As a consequence, we obtain the following refinement of the inequality \eqref{eq:481} for $u_0 \in \Leven(\R)$.
\begin{corollary} \label{cor:481}
	Let $u_0 \in \Leven(\R)$, and let 
	\begin{equation}
		f(x) = (u_0 \otimes u_0 \otimes u_0 \otimes u_0)(x)  .
	\end{equation}
	Then we have
	\begin{equation} 
\norm{ \norm{ e^{i t \Laplacian} u_0(x) }_{L^4(\R_x)} }_{L^8(\R_t)}^8
	\leq \frac{1}4 \mleft( \norm{u_0}_{L^2(\R)}^8 - \frac{17 - \sqrt{97}}{32} \distfX[f][\Lrad(\R^4)]^2 \mright).
	\end{equation}
\end{corollary}
Moreover, using Corollary \ref{cor:481}, we also obtain the following.
\begin{corollary} \label{cor:extremizing sequence}
	Let $\{ u_n \}_{n \in \N} \subset \Leven(\R) \setminus \{0\}$ be an extremizing sequence of the inequality \eqref{eq:481}, that is, 
	\begin{equation}
	\lim_{n \to \infty} \mleft( \frac{ \norm{ \norm{ e^{i t \Laplacian} u_n(x) }_{L^4(\R_x)} }_{L^8(\R_t)} }{ \norm{u_n}_{L^2(\R)} } \mright)^8 = \frac{1}{4} .
	\end{equation}
	Then we have
	\begin{equation} 
	\lim_{n \to \infty} \frac{ \distfX[u_n][\Gaussians] }{ \norm{u_n}_{L^2(\R)} } = 0 , 
	\end{equation}
	where
	\begin{equation}
\Gaussians \coloneqq \set{ \R \ni x \longmapsto C\exp(Ax^2) }{ A,C \in \C, \, \Re A < 0 } .
	\end{equation}
\end{corollary}
This follows from Corollary \ref{cor:481} and
\begin{align}
\distfX[u_n \otimes u_n \otimes u_n \otimes u_n][\Lrad(\R^4)]^2
&\geq \norm{u_n}_{L^2(\R)}^4 \distfX[u_n \otimes u_n][\Lrad(\R^2)]^2 \\ 
&\gtrsim \norm{u_n}_{L^2(\R)}^6 \distfX[u_n][\Gaussians]^2 .
\end{align}
The first inequality comes from $\Lrad(\R^4) \subset \Lrad(\R^2) \otimes \Lrad(\R^2)$, and the second one is due to \citet[Theorem 1.1]{Chr2015}.
We remark that \citet[Theorem 1.4]{GNO2026} recently established an analogue of Corollary \ref{cor:extremizing sequence} under the non-negativity assumption for $u_0$ instead of the evenness.
\subsection*{Organization of the paper}
In Section \ref{section:preliminaries}, we introduce some notation and basic facts as preliminaries.
In Section \ref{section:equivalent condition}, we give an equivalent characterization of Conjecture \ref{conj:481}.
We prove Conjecture \ref{conj:481} with a non-optimal constant $\beta = 3/16$ in Section \ref{section:proof}, and with the optimal constant $\beta = (17 - \sqrt{97})/32$ in Section \ref{section:optimal constant}.  

\section{Preliminaries} \label{section:preliminaries}
\subsection{\texorpdfstring{$L^2$}{L²} functions invariant under orthogonal transformations} \label{section:invariant functions}
First, we recall that $\Lrad(\R^d)$ and $\Leven(\R^d)$ are the subspaces of $L^2(\R^d)$ consisting of all radially symmetric functions and of all \emph{coordinate-wise} even functions, as we defined in Section \ref{section:introduction}.
We write $\Projrad$ and $\Projeven$ for the orthogonal projections onto these spaces, respectively.
It is easy to see that we have
\begin{gather}
\Projrad f(x) = \int_{\omega \in \S^{d-1}} f( \abs{x} \omega ) \, dS(\omega) ,
\label{eq:Projrad} \\
\Projeven f(x) = \frac{1}{2^d} \sum_{ \varepsilon \in \{ -1, 1 \}^d } f( \varepsilon_1 x_1, \varepsilon_2 x_2, \ldots, \varepsilon_d x_d ) ,
\label{eq:Projeven} \\
\distfX[f][\Lrad(\R^d)] = \norm{ (1 - \Projrad) f }_{L^2(\R^d)} ,
\label{eq:dinstance formula} \noeqref{eq:dinstance formula}
\end{gather}
where $d S$ denotes the probability spherical surface measure of $\S^{d-1}$ in \eqref{eq:Projrad}.

Now let $V \subset \R^d$ be a $(d-2)$-dimensional linear subspace. 
We write $\LV(\R^d)$ for the subspace of $L^2(\R^d)$ consisting of all functions satisfying
\begin{equation}
f(\Rotate x) = f(x)
\end{equation}
for every rotation $\Rotate \in \SO(d)$ which fixes $V$ (i.e., $\Rotate x = x$ for every $x \in V$), and $\ProjV$ denotes the corresponding orthogonal projection. 
Note that this can be written as the rotation average
\begin{equation} \label{eq:ProjV}
	\ProjV f(x) = \frac{1}{2\pi} \int_{\theta \in \intervalco{0}{2\pi}} f( \RotateV(\theta) x) \, d\theta, 
\end{equation}
where $\RotateV(\theta) \in \SO(d)$ is the rotation through angle $\theta$ which fixes the plane $V$. 
For instance, when $d = 4$ and $\planeV_0 \coloneqq \Span \{ (1, 0, 0, 0), (0, 0, 1, 0) \}$, 
we have
\begin{equation} \label{eq:ProjV for V0}
\ProjV[\planeV_0] f(x) = \frac{1}{2\pi} \int_{\theta \in \intervalco{0}{2\pi}} f(x_1, \sqrt{x_2^2 + x_4^2} \cos{\theta}, x_3 , \sqrt{x_2^2 + x_4^2} \sin{\theta}) \, d\theta .
\end{equation}

Observe that we have
\begin{equation}
\Lrad \subset \Leven \cap \LV ,
\end{equation}
so that 
\begin{gather} 
\Projrad \Projeven = \Projeven \Projrad = \Projrad , 
\label{eq:commute Projrad Projeven} \noeqref{eq:commute Projrad Projeven} \\
\Projrad \ProjV = \ProjV \Projrad = \Projrad
\label{eq:commute Projrad ProjV} \noeqref{eq:commute Projrad ProjV} 
\end{gather}
for each $(d-2)$-dimensional subspace $V \subset \R^d$.
It is also easy to see that 
\begin{equation} \label{eq:ProjV changing V} 
\ProjV[\orthogonalE \planeV]f(E x) = \ProjV[\planeV] [f \circ \orthogonalE] ( x )
\end{equation}
holds for every orthogonal matrix $\orthogonalE \in O(d)$.
When $d=4$, one can compute $\ProjV f$ combining
\eqref{eq:ProjV for V0} and \eqref{eq:ProjV changing V}.

An important fact is that, when $d = 4$, the orthogonal complement of each $2$D plane $V$ is also a $2$D plane, so that $\ProjVperp$ is well-defined.
Moreover, we have
\begin{equation} \label{eq:commute ProjV ProjVperp}
\ProjVperp \ProjV[V] = \ProjV \ProjVperp .
\end{equation}
To see this, observe that $\RotateV(\theta_1)$ and $\RotateVperp(\theta_2)$ commute for all rotation angles $\theta_1, \theta_2$.
In fact, for each $x \in \R^4$, we have
\begin{equation}
\RotateV(\theta_1) \RotateVperp(\theta_2) x
= \RotateVperp(\theta_2) \RotateV(\theta_1) x
= \RotateVperp(\theta_2) x_V + \RotateV(\theta_1) x_{\perp} , 
\end{equation}
where $x_V$ and $x_\perp$ denote the orthogonal projections of $x$ onto $V$ and $V^\perp$, respectively.
Therefore, by \eqref{eq:ProjV}, we have 
\begin{align} \label{eq:ProjV ProjVperp formula}
&\quad \ProjV \ProjVperp f(x) = \ProjVperp \ProjV f(x) \\
& = \frac{1}{4\pi^2} \int_{\theta_2 \in \intervalco{0}{2\pi}} \int_{\theta_1 \in \intervalco{0}{2\pi}} f( \RotateVperp(\theta_2) x_V + \RotateV(\theta_1) x_{\perp} ) \, d\theta_1 \, d\theta_2 .
\end{align}

We remark that the averaging formulas \eqref{eq:Projrad}, \eqref{eq:Projeven}, and \eqref{eq:ProjV} also define orthogonal projections on $L^2(\S^{d-1})$.
By an abuse of notation, we use the same symbols $\Projrad$, $\Projeven$, and $\ProjV$ for those projections.
Then the commuting formulas \eqref{eq:commute Projrad Projeven}, \eqref{eq:commute Projrad ProjV}, and \eqref{eq:ProjV ProjVperp formula} are also valid on $L^2(\S^{d-1})$.

In what follows, we mainly use the polar coordinates
\begin{align} \label{eq:Koornwinder coordinates}
	\omega(\varphi, \psi, \theta) =
	\begin{pmatrix}
		\cos \psi \cos \theta \\
		\cos \varphi \sin \theta \\
		\sin \psi \cos \theta \\
		\sin \varphi \sin \theta
	\end{pmatrix} 
\end{align}
on $\S^3$ rather than the Cartesian coordinates.
When no confusion is possible, we write $f(\varphi, \psi, \theta)$ for $f(\omega(\varphi, \psi, \theta))$.
Then we have
\begin{align}
	\Projeven f(\omega) &= \frac{1}{16} \sum_{ \varepsilon \in \{ -1, 1 \}^2 } \sum_{ j \in \{0, 1\}^2 } f( \varepsilon_1 \varphi + j_1 \pi , \varepsilon_2 \psi + j_2 \pi , \theta ) , 
	\label{eq:Projeven polar}\\
	\ProjV[\planeV_0] f(\omega) &= \frac{1}{2\pi} \int_{\varphi \in \intervalco{0}{2\pi}} f(\varphi, \psi, \theta) \, d\varphi , 
	\label{eq:ProjV polar}\\
	\ProjV[\planeVperp_0] f(\omega) &= \frac{1}{2\pi} \int_{\psi \in \intervalco{0}{2\pi}} f(\varphi, \psi, \theta) \, d\psi
	\label{eq:ProjVperp polar}
\end{align}
for $\planeV_0 = \Span \{ (1, 0, 0, 0), (0, 0, 1, 0) \}$.

\subsection{Spherical harmonics} \label{subsection:spherical harmonics}
We assume $d \geq 2$ in this section.
Let $\Hk(\S^{d-1})$ be the space of spherical homogeneous harmonic polynomials of degree $k$ for each $k \in \N = \{0, 1, 2,  \ldots \}$.
It is well known that the orthogonal decomposition
\begin{equation} \label{eq:spherical harmonic decomposition on sphere}
L^2(\S^{d-1}) = \bigoplus_{k \in \N} \Hk(\S^{d-1})
\end{equation}
holds. 
As a consequence, using polar coordinates $x = r \omega$, $(r, \omega) \in \intervaloo{0}{\infty} \times \S^{d-1}$ and letting $\Lk (\R^d) \coloneqq  L^2( \intervaloo{0}{\infty}, r^{d-1} \, dr ) \otimes \Hk(\S^{d-1})$, we obtain 
\begin{equation} \label{eq:spherical harmonic decomposition on R^d}
	L^2(\R^d) = \bigoplus_{k \in \N} \Lk (\R^d) .
\end{equation}
We write $\ProjHk$ for the orthogonal projection onto $\Lk$ for each $k \in \N$. Note that $\Lk[0]$ = $\Lrad$, so that $\ProjHk[0] = \Projrad$.
Moreover, $\Lk(\R^d)$ is invariant with respect to both $\Projeven$ and $\ProjV$ due to \eqref{eq:Projeven} and \eqref{eq:ProjV}, so that
\begin{gather} 
\Projeven \ProjHk[k] = \ProjHk[k] \Projeven = \ProjHk[k] \Projeven \ProjHk[k] , 
\label{eq:commute Projeven ProjHk} \noeqref{eq:commute Projeven ProjHk} \\
\ProjV \ProjHk[k] = \ProjHk[k] \ProjV = \ProjHk[k] \ProjV \ProjHk[k] 
\label{eq:commute ProjV ProjHk} \noeqref{eq:commute ProjV ProjHk}
\end{gather}
for each $k \in \N$.

We remark that these identities \eqref{eq:commute Projeven ProjHk} and \eqref{eq:commute ProjV ProjHk} are also valid on $L^2(\S^{d-1})$ by understanding $\ProjHk[k]$ as the orthogonal projection from $L^2(\S^{d-1})$ onto $\Hk(\S^{d-1})$. 
As a consequence, $\Projeven$ and $\ProjV$ can be understood as orthogonal projections on $\Hk(\S^{d-1})$ for each $k \in \N$. 
In order to study these operators on $\Hk(\S^3)$, the following orthonormal basis given by \citet{Koo1973} is useful. 
\begin{proposition}[{\cite[Theorem 4.2]{Koo1973}}] \label{prop:Koornwinder}
	For each $(k, m, n) \in \N^3$ satisfying $(k - m - n)/2 \in \N$ and for $(a, b) \in \{ 0, 1 \}^2$, we define
	\begin{equation}
	\begin{aligned} 
	&\quad \Ykmnab(\omega) \\
	&= \normalizeN \Re( (-i)^a (\omega_2 + i \omega_4)^m ) \Re( (-i)^b (\omega_1 + i \omega_3)^n ) \\
	&\quad \times \JacobiP[(k-m-n)/2][m][n]( \omega_1^2 + \omega_3^2 - \omega_2^2 - \omega_4^2 ) \\ 
	&= \normalizeN (\sin\theta)^m (\cos\theta)^n \cos(m \varphi - a \pi / 2) \cos(n \psi - b \pi / 2) \\
	&\quad \times \JacobiP[(k-m-n)/2][m][n](\cos{2\theta}) ,
	\end{aligned} \label{eq:Koornwinder functions}
		\end{equation}
	where $\{ \JacobiP \}_{k \in \N}$ denotes the Jacobi polynomials, and $\normalizeN \in \intervaloo{0}{\infty}$ is a positive normalizing constant.
Then the functions
\begin{equation} \label{eq:Koornwinder basis}
\Set{ \Ykmnab }{ \begin{matrix}
 (m, n) \in \N^2 , \, (a, b) \in \{ 0, 1 \}^2 , \\
  (k - m - n)/2 \in \N, \\
  m = 0 \implies a = 0, \\
  n = 0 \implies b = 0
	\end{matrix} }
\end{equation}
form an orthonormal basis of $\Hk(\S^3)$ for each $k \in \N$. 
	In addition, the normalizing constant can be written as
\begin{equation} \label{eq:normalizeN}
	\normalizeN = 
	2 \mleft(
	\frac{k+1}{2^{ \delta(m) + \delta(n) }} \cdot
	\frac{ \factorial{j} \, \factorial{(j+m+n)} }{ \factorial{(j+m)} \, \factorial{(j+n)} }
	\mright)^{1/2} , \quad
	j \coloneqq (k - m - n)/2 , 
\end{equation}
where $\delta$ denotes the discrete delta function.
\end{proposition}
The definition and basic properties of the Jacobi polynomials $\{ \JacobiP \}_{k \in \N}$ can be found in \citet[\href{https://dlmf.nist.gov/18}{Chapter 18}]{DLMF} and \citet[\textsection 7.39, \textsection 8.96]{GR2014}, for example. 
We remark that the Jacobi polynomials coincide with the Gegenbauer polynomials $\{ \GegenbauerC \}_{k \in \N}$ when $\alpha = \beta = \mu - 1/2$ (up to a constant depending on the normalization convention).
The Gegenbauer polynomials play an important role in the study of optimal constants of smoothing estimates for the free Schr\"{o}dinger and Dirac equations; see \citet*{BSS2015} and \citet{Suz2025, Suz2025_2}.
\subsection{A mean of four orthogonal projections} \label{subsection:projections}
The following elementary observations play an essential role in our proof of Theorem \ref{thm:481}.
\begin{proposition} \label{prop:functional calculus}
	Let $H$ be a Hilbert space, let $H_0$ be a subspace of $H$, and let $\{P_j\}_{j=1}^{4}$ be orthogonal projections on $H$ satisfying
	\begin{equation} \label{eq:commuting pairs}
	P_1 P_2 = P_2 P_1, \quad P_3 P_4 = P_4 P_3 .
	\end{equation}
	We define $T \colon H \to H$ by
	\begin{equation}
		T \coloneqq \frac{1}{4} \sum_{j=1}^{4} P_j .
	\end{equation}
	Then the following hold.
	\begin{enumerate}[label=\textup{(P\ref{prop:functional calculus}.\roman*)}]
		\item \label{item:functional calculus bound 3+}
		If $(P_1 P_2, P_3 P_4) \ne (0, 0)$, then we have
		\begin{equation} \label{eq:functional calculus bound 3+}
			\norm{ T }_{H \to H} \leq \frac{1}{4} ( 3 + \norm{ P_1 P_2 P_3 P_4 }_{H \to H} ).
		\end{equation}	
		\item \label{item:functional calculus bound 2}
		If $P_1 P_2 = P_3 P_4 = 0$ on $H_0$, then we have
		\begin{equation} \label{eq:functional calculus bound 2}
		\norm{ T }_{H_0 \to H} \leq \frac{1}{2} .
		\end{equation}
		\item \label{item:functional calculus bound 3}
If 
		\begin{equation}
		\Re \innerproduct{P_1 P_2 f}{P_3 P_4 f} \leq 0
		\end{equation}
		holds for every $f \in H_0$, then we have
		\begin{equation} \label{eq:functional calculus bound 3}
		\norm{ T }_{H_0 \to H} \leq \frac{3}{4} .
		\end{equation}	
	\end{enumerate}
\end{proposition}
\begin{proof}[Proof of Proposition \ref{prop:functional calculus}]
Since $(P_1, P_2)$ and $(P_3, P_4)$ are commuting pairs of orthogonal projections, the operators
\begin{equation}
1 - (1 - P_1) (1 - P_2) , \quad 1 - (1 - P_3) (1 - P_4)
\end{equation}
are orthogonal projections, and the identity
\begin{equation}
4T = ( 1 - (1 - P_1) (1 - P_2) ) + ( 1 - (1 - P_3) (1 - P_4) ) + P_1 P_2 + P_3 P_4 
\end{equation}
holds.
Thus, we have
\begin{equation} \label{eq:functional calculus ineq}
4 \norm{T}_{H_0 \to H} \leq 2 + \norm{ P_1 P_2 + P_3 P_4  }_{H_0 \to H} .
\end{equation}
This bound immediately implies \ref{item:functional calculus bound 2}. 
Similarly, \ref{item:functional calculus bound 3+} follows from \eqref{eq:functional calculus ineq} and the Duncan--Taylor formula \cite[Theorem 7]{DT1976}
\begin{equation}
\norm{ P_1 P_2 + P_3 P_4 }_{H \to H} = 1 + \norm{ P_1 P_2 P_3 P_4 }_{H \to H} ,
\end{equation}
which is valid because $P_1 P_2$ and $P_3 P_4$ are orthogonal projections on $H$ satisfying $(P_1 P_2, P_3 P_4) \ne (0, 0)$.
In order to prove \ref{item:functional calculus bound 3}, observe that
\begin{align}
\norm{ (P_1 P_2 + P_3 P_4) f }_H^2
&= \norm{ P_1 P_2 f }_H^2 + 2 \Re \innerproduct{P_1 P_2 f}{P_3 P_4 f} + \norm{ P_3 P_4 f }_H^2 \\
&\leq \innerproduct{P_1 P_2 f}{ f }_{H} + \innerproduct{P_3 P_4 f}{ f }_{H} \\
&\leq \norm{ (P_1 P_2 + P_3 P_4) f }_H \norm{ f }_H
\end{align}
holds for every $f \in H_0$ under the assumption of \ref{item:functional calculus bound 3}, so that
\begin{equation}
\norm{P_1 P_2 + P_3 P_4}_{H_0 \to H} \leq 1 .
\end{equation}
Combining this with \eqref{eq:functional calculus ineq}, we obtain \ref{item:functional calculus bound 3}.
\end{proof}

\section{An equivalent condition for Conjecture \ref{conj:481}} \label{section:equivalent condition}
For each $(\sigma,\tau) \in \{-1,1\}^2$, we define the $2$D plane $V_{\sigma, \tau} \subset \R^4$ by
\begin{equation} \label{eq:plane Vst}
	\Vst \coloneqq \Span \mleft \{ \frac{1}{\sqrt{2}} (1,0,\sigma,0), \frac{1}{\sqrt{2}} (0,1,0,\tau) \mright \}.
\end{equation}
For simplicity, we write $\ProjV[\sigma, \tau] \coloneqq \ProjV[V_{\sigma, \tau}]$.
We begin with the following result given by \citet{HZ2006}.
\begin{proposition}[{\cite[Theorem 1.3]{HZ2006}}] \label{prop:quadoratic form}
	Let $f \in L^2(\R^4)$, and let 
	\begin{equation}
		U(x, t) \coloneqq (e^{i t \Laplacian} f ) (x_1, x_2, x_1, x_2) .
	\end{equation}
	 Then we have
	\begin{equation} \label{eq:projecton formula}
		\norm{U}_{L^2(\R^2_x \times \R_t)}^2
		= \frac{1}{4} \innerproduct{ \ProjV[1, 1] f }{f}_{L^2(\R^4)} .
	\end{equation}
\end{proposition}
To be precise, \citet{HZ2006} state the identity \eqref{eq:projecton formula} only for $f = u_0 \otimes u_0$ with $u_0\in L^2(\R^2)$. 
Nevertheless, their proof also works for arbitrary $f\in L^2(\R^4)$. See also \citet[Lemma 6]{Car2009} and \citet[p.~1712]{GZ2022}.
An immediate consequence is that the inequality \eqref{eq:481 improved} holds with 
\begin{equation} \label{eq:conj 481 equivalent P11}
	\beta = 1 - \norm{\Projeven \ProjV[1, 1] (1-\Projrad) \Projeven}_{L^2(\R^4) \to L^2(\R^4)} 
\end{equation}
for every $f \in \Leven(\R^4)$.
Therefore, Conjecture \ref{conj:481} is equivalent to 
\begin{equation}
	\norm{\Projeven \ProjV[1, 1] (1-\Projrad) \Projeven}_{L^2(\R^4) \to L^2(\R^4)} < 1.
\end{equation}
However, it is not straightforward to estimate this complicated operator.
Our aim in this section is to show the following result, which allows us to bound the operator norm via Proposition \ref{prop:functional calculus}.
\begin{proposition} \label{prop:Projmean identity}
	We define $\Projmean \colon L^2(\R^4) \to L^2(\R^4)$ by
	\begin{equation} \label{eq:Projmean def}
	\Projmean \coloneqq \frac{1}{4} \sum_{(\sigma,\tau) \in \{-1,1\}^2} \ProjV[\sigma, \tau] .
	\end{equation}
	Then we have
	\begin{gather} 
	\Projeven \Projmean = \Projmean \Projeven = \Projeven \ProjVst \Projeven , \quad \forall (\sigma, \tau) \in \{-1, 1\}^2 ,
	\label{eq:commute Projeven Projmean} \\
	\Projrad \Projmean = \Projmean \Projrad = \Projrad ,
	\label{eq:commute Projrad Projmean} \\
		\ProjHk \Projmean = \Projmean \ProjHk = \ProjHk \Projmean \ProjHk .
	\label{eq:commute ProjHk Projmean} 
	\end{gather}
	As a consequence, the inequality \eqref{eq:481 improved} holds with 
	\begin{equation} \label{eq:481 improved abstract}
		\beta = 1 - \norm{\Projmean (1-\Projrad)}_{\Leven(\R^4) \to \Leven(\R^4)} 
	\end{equation}
	for every $f \in \Leven(\R^4)$, and this is the largest possible value of $\beta$.
\end{proposition}
\begin{proof}[Proof of Proposition \ref{prop:Projmean identity}]
	For each $\varepsilon \in \{-1, 1\}^4$, we write $A_\varepsilon \coloneqq \operatorname{diag}(\varepsilon_1,\ldots,\varepsilon_4)$
	and define $\mathcal{A}_\varepsilon \colon L^2(\R^4) \to L^2(\R^4)$ by $\mathcal{A}_\varepsilon f(x) \coloneqq f(A_\varepsilon x)$.
	Then we have 
	\begin{align}
	A_\varepsilon V_{\sigma,\tau} &= V_{\varepsilon_1 \varepsilon_3 \sigma,\, \varepsilon_2 \varepsilon_4 \tau}
\shortintertext{due to \eqref{eq:plane Vst}, which implies}
	\mathcal{A}_\varepsilon \ProjV[\sigma, \tau] \mathcal{A}_\varepsilon &= \ProjVst[\varepsilon_1 \varepsilon_3 \sigma][\varepsilon_2 \varepsilon_4 \tau] 
	\label{eq:APA = P}
	\end{align}
	for each $\varepsilon \in \{-1, 1\}^4$ and $(\sigma, \tau) \in \{-1, 1\}^2$.
On the other hand, we have
	\begin{equation}
	\Leven(\R^4) = \set{ f \in L^2(\R^4) }{ \forall \varepsilon \in \{-1, 1\}^4 , \quad \mathcal{A}_\varepsilon f = f } ,
	\end{equation} 
	so that
	\begin{equation} \label{eq:commute A Projeven}
	\mathcal{A}_\varepsilon \Projeven = \Projeven \mathcal{A}_\varepsilon = \Projeven .
	\end{equation}
	Combining \eqref{eq:APA = P} and \eqref{eq:commute A Projeven}, we get
	\begin{equation}
	\mathcal{A}_\varepsilon \ProjV[\sigma, \tau] \Projeven = \ProjVst[\varepsilon_1 \varepsilon_3 \sigma][\varepsilon_2 \varepsilon_4 \tau] \Projeven , 
	\end{equation}
	and averaging this with respect to $\varepsilon \in \{-1, 1\}^4$ for fixed $(\sigma, \tau) \in \{-1, 1\}^2$ yields
	\begin{equation}
	\Projeven \ProjV[\sigma, \tau] \Projeven = \Projmean \Projeven .
	\end{equation}
	Similarly, we obtain
	\begin{equation}
		\Projeven \ProjV[\sigma, \tau] \Projeven = \Projeven \Projmean 
	\end{equation}
	by averaging
	\begin{equation}
	\Projeven \ProjV[\sigma, \tau] \mathcal{A}_\varepsilon = \Projeven \ProjVst[\varepsilon_1 \varepsilon_3 \sigma][\varepsilon_2 \varepsilon_4 \tau] .
	\end{equation}
	This proves \eqref{eq:commute Projeven Projmean}.
	
	The other commutation relations \eqref{eq:commute Projrad Projmean} and \eqref{eq:commute ProjHk Projmean} are immediate from \eqref{eq:commute Projrad ProjV} and \eqref{eq:commute ProjV ProjHk}.
	
	Now we prove \eqref{eq:481 improved abstract}.
	Let $f \in \Leven(\R^4)$ and write $g = \Projrad f$, $h = (1 - \Projrad)f$. 
	Then we have
	\begin{gather}
			\norm{ \ProjVst[1][1] f }_{L^2(\R^4)}^2 \underset{\eqref{eq:commute Projeven Projmean}}{=} \innerproduct{\Projmean f}{f}_{L^2(\R^4)} , \\
		\Projmean g = \Projmean \Projrad f \underset{\eqref{eq:commute Projrad Projmean}}{=} \Projrad f = g , \\
		\Projmean h = \Projmean (1 - \Projrad) \Projeven h , \\
		\innerproduct{g}{\Projmean  h}_{L^2} = \innerproduct{\Projmean g}{h}_{L^2} = \innerproduct{g}{h}_{L^2} = 0 ,
	\end{gather}
	so that
	\begin{align}
		\norm{ \ProjVst[1][1] f }_{L^2(\R^4)}^2
		&= \innerproduct{\Projmean g}{g}_{L^2} + \innerproduct{\Projmean g}{h}_{L^2} + \innerproduct{\Projmean h}{g}_{L^2} + \innerproduct{\Projmean h}{h}_{L^2} \\
		&\leq \norm{g}_{L^2}^2 + 0 + 0 + \norm{\Projmean (1-\Projrad) \Projeven}_{L^2 \to L^2} \norm{h}_{L^2}^2 \\
		&=\norm{f}_{L^2}^2 - (1 - \norm{\Projmean (1-\Projrad)}_{\Leven \to \Leven} ) \norm{h}_{L^2}^2 .
	\end{align}
	This shows that the inequality \eqref{eq:481 improved} holds with $\beta$ given in \eqref{eq:481 improved abstract}, which is sharp because
	\begin{equation*}
		\sup_{h \in (1 - \Projrad) \Leven \setminus \{0\}} \frac{\innerproduct{\Projmean (1-\Projrad) \Projeven h}{h}_{L^2}}{ \norm{h}_{L^2}^2 } = \norm{\Projmean (1-\Projrad)}_{\Leven \to \Leven} . \qedhere
	\end{equation*}
\end{proof}
Proposition \ref{prop:Projmean identity} reveals that Conjecture \ref{conj:481} is equivalent to
\begin{align}
\norm{\Projmean (1-\Projrad)}_{\Leven(\R^4) \to \Leven(\R^4)}  &< 1 .
\label{eq:conj 481 equivalent Q}
\shortintertext{
We are going to prove the upper bound
}
\norm{\Projmean (1-\Projrad)}_{\Leven(\R^4) \to \Leven(\R^4)}  &\leq \frac{13}{16}
	\label{eq:481 norm upper bound even} \noeqref{eq:481 norm upper bound even}
\shortintertext{
in Section \ref{section:proof}, and to give the exact value
}
\norm{\Projmean (1-\Projrad)}_{\Leven(\R^4) \to \Leven(\R^4)}  &= \frac{15 + \sqrt{97}}{32}
	\label{eq:481 norm exact value}
\end{align}
in Section \ref{section:optimal constant}.

\section{Proof of Conjecture \ref{conj:481}} \label{section:proof}
Our aim in this section is to prove the following upper bounds.
\begin{theorem} \label{thm:Projmean ProjHk upperbound}
	We define a sequence $\{ \lambdak \}_{k \in \N}$ by
	\begin{equation} \label{eq:lambdak}
		\lambdak[k] = \mleft( 2^{-2k} \binom{2k}{k} \mright)^2 .
	\end{equation}
	Then we have
	\begin{align} 
		\norm{\Projmean}_{\Hk[2k+1](\S^{3}) \to \Hk[2k+1](\S^3)} &\leq 1/2 , 
		\label{eq:Projmean ProjHk upperbound 2k+1} \\
		\norm{\Projmean}_{\Hkeven[4k+2](\S^{3}) \to \Hkeven[4k+2](\S^3)} &\leq 1/2 , 
		\label{eq:Projmean ProjHk upperbound 4k+2 even} \\
		\norm{\Projmean}_{\Hkeven[4k+2]^\perp(\S^{3}) \to \Hkeven[4k+2]^\perp(\S^3)} &\leq 3/4 , 
		\label{eq:Projmean ProjHk upperbound 4k+2 evenperp} \\		 
		\norm{\Projmean}_{\Hkeven[4k](\S^{3}) \to \Hkeven[4k](\S^3)} &\leq \frac{1}{4} (3 + \lambdak) , 
		\label{eq:Projmean ProjHk upperbound 4k even} \\
		\norm{\Projmean}_{\Hkeven[4k]^\perp(\S^{3}) \to \Hkeven[4k]^\perp(\S^3)} &\leq 3/4 ,
		\label{eq:Projmean ProjHk upperbound 4k evenperp} 
	\end{align}
	where $\Hkeven[j](\S^{3}) \coloneqq \Projeven \Hk[j](\S^{3})$, and $\Hkeven[j]^\perp(\S^{3})$ denotes its orthogonal complement taken in $\Hk[j](\S^{3})$.
	As a consequence, we have
	\begin{equation} 	\label{eq:481 norm upper bound}
		\norm{\Projmean(1 - \Projrad)}_{L^2(\R^4) \to L^2(\R^4)} \leq \frac{13}{16} ,
	\end{equation}
	so that Conjecture \ref{conj:481} holds with the constant $\beta = 3/16$.
\end{theorem}
First, we prove the following using Proposition \ref{prop:Koornwinder}.
\begin{proposition} \label{prop:koornwinder p=q=2}
	Let $\basisE = (\unitvecj[1], \unitvecj[2], \unitvecj[3], \unitvecj[4])$ be a $4 \times 4$ orthogonal matrix, and let $\planeV \coloneqq \Span\{ \unitvecj[1], \unitvecj[3] \}$. 
	Then we have the following for each $k \in \N$.
	\begin{enumerate}[label=\textup{(P\ref{prop:koornwinder p=q=2}.\roman*)}]
		\item \label{item:Koornwinder odd}
		On $\Hk[2k + 1]$, we have
		\begin{equation}
			\ProjV \ProjVperp = 0.
		\end{equation}
		\item \label{item:Koornwinder even}
		On $\Hk[2k]$, $\ProjV \ProjVperp$ is the orthogonal projection onto the one-dimensional subspace spanned by
		\begin{equation}
			\S^{3} \ni \omega \longmapsto \Ykmnab[2k][0][0][0][0](\orthogonalE^{-1} \omega) .
		\end{equation}
	\end{enumerate}
\end{proposition}
\begin{proof}[Proof of Proposition \ref{prop:koornwinder p=q=2}]
Without loss of generality, we assume that $\basisE$ is the identity matrix thanks to \eqref{eq:ProjV changing V}.
Then substituting \eqref{eq:Koornwinder functions} into \eqref{eq:ProjV polar} and \eqref{eq:ProjVperp polar} yields
\begin{align}
	\ProjV[\planeV_0] \Ykmnab
	 = 
	\begin{cases}
	\Ykmnab[k][0][n] , & m = 0 , \\
	0 , & \text{otherwise} , 
	\end{cases}  \\
	\ProjV[\planeV_0^\perp] \Ykmnab
	= 
	\begin{cases}
	\Ykmnab[k][m][0] , & n = 0 , \\
	0 , & \text{otherwise} , 
\end{cases} 
\shortintertext{so that}
	\ProjV[\planeV_0] \ProjV[\planeV_0^\perp] \Ykmnab
	= 
		\begin{cases}
		\Ykmnab[k][0][0] , & m = n = 0 , \\
		0 , & \text{otherwise} 
	\end{cases} 
\end{align}
for each basis function $\Ykmnab \in \Hk$. 
On the other hand, $m = n = 0$ is allowed if and only if $k$ is even and $a = b = 0$ according to \eqref{eq:Koornwinder basis}. This proves \ref{item:Koornwinder odd} and \ref{item:Koornwinder even}.
\end{proof}
\begin{proof}[Proof of Theorem \ref{thm:Projmean ProjHk upperbound}]
Observe that $V_{\sigma, \tau}^\perp = V_{- \sigma, - \tau}$ holds for each $(\sigma, \tau) \in \{ -1, 1 \}^2$. Therefore, we get
\begin{equation}
\ProjV[1, 1] \ProjV[-1, -1] = \ProjV[-1, -1] \ProjV[1, 1] , \quad 
\ProjV[1, -1] \ProjV[-1, 1] = \ProjV[-1, 1] \ProjV[1, -1]
\end{equation}
thanks to \eqref{eq:commute ProjV ProjVperp}.
This allows us to apply Proposition \ref{prop:functional calculus} to $\Projmean$.

On $\Hk[2k+1]$, we have
\begin{equation}
\ProjV[1, 1] \ProjV[-1, -1] = \ProjV[1, -1] \ProjV[-1, 1] = 0
\end{equation}
due to \ref{item:Koornwinder odd}, so that \eqref{eq:Projmean ProjHk upperbound 2k+1} follows from \ref{item:functional calculus bound 2}. 

In order to find the other bounds, we define the orthogonal matrix $\orthogonalEst$ by
\begin{equation} \label{eq:orthogonalEst}
	\orthogonalEst
	=\frac{1}{\sqrt{2}}
	\begin{pmatrix}
		1&1&0&0\\
		0&0&1&1\\
		\sigma&-\sigma&0&0\\
		0&0&\tau&-\tau
	\end{pmatrix} , 
\end{equation}
so that $\Vst = \orthogonalEst \planeV_0$.
Next we introduce the functions $\Zk, \Wk \in \Hk[2k](\S^3)$ given by
\begin{align}
	\Zk(\omega) &\coloneqq \Ykmnab[2k][0][0][0][0]( \orthogonalEst[1][1]^{-1} \omega ) \\
	&= \normalizeN[2k][0][0] \LegendreP[k] (\sin{2\psi}\cos^2{\theta} + \sin{2\varphi} \sin^2{\theta}) ,
	\label{eq:Zk} \\ 
	\Wk(\omega) &\coloneqq \Ykmnab[2k][0][0][0][0]( \orthogonalEst[1][-1]^{-1} \omega ) \\
	& = \normalizeN[2k][0][0] \LegendreP[k] (\sin{2\psi} \cos^2{\theta} - \sin{2\varphi} \sin^2{\theta}) .
	\label{eq:Wk}
\end{align}
Then \ref{item:Koornwinder even} states that $\ProjVst[1][1] \ProjVst[-1][-1]$ and $\ProjVst[1][-1] \ProjVst[-1][1]$ are the orthogonal projections from $\Hk[2k]$ onto the one-dimensional subspaces $\Span\{ \Zk \}$ and $\Span\{ \Wk \}$, in other words, 
\begin{align}
\ProjVst[1][1] \ProjVst[-1][-1] f(\omega) &=  \innerproduct{f}{\Zk}_{L^2(\S^3)} \Zk(\omega) ,
	\label{eq:Projection Zk} \\ 
\ProjVst[1][-1] \ProjVst[-1][1] f(\omega) &= \innerproduct{f}{\Wk}_{L^2(\S^3)} \Wk(\omega) 
	\label{eq:Projection Wk}
\end{align}
for every $f \in \Hk[2k]$ and $\omega \in \S^3$.
Now we claim the following:
\begin{gather}
	\Projeven \Zk[2k+1] = \Projeven \Wk[2k+1] = 0,
		\label{eq:Zk Wk Projeven 2k+1} \\
	\Projeven \Zk[2k] = \Projeven \Wk[2k] = \tfrac{1}{2} (\Zk[2k] + \Wk[2k]) ,
	\label{eq:Zk Wk Projeven 2k} \\
	\innerproduct{\Zk[2k+1]}{\Wk[2k+1]}_{L^2(\S^3)} = 0,
	\label{eq:Zk Wk inner product 2k+1} \\
	\innerproduct{\Zk[2k]}{\Wk[2k]}_{L^2(\S^3)}
	=\lambdak  .
	\label{eq:Zk Wk inner product 2k}
\end{gather}
For the first and second ones, we use the parity identity \cite[8.961.1]{GR2014}
\begin{equation} \label{eq:parity identity}
\LegendreP(-x) = (-1)^k \LegendreP(x) .
\end{equation}
Then we obtain
\begin{align}
\Projeven \Zk (\omega) 
\underset{\eqref{eq:Projeven polar}}&{=} 
\frac{1}{4} \normalizeN[2k][0][0] \sum_{\varepsilon \in \{-1, 1\}^2}
\JacobiP[k][0][0](\varepsilon_1 \sin{2\psi} \cos^2{\theta} + \varepsilon_2 \sin{2\varphi} \sin^2{\theta}) \\
\underset{\eqref{eq:parity identity}}&=
\frac{1}{4} ( \Zk(\omega) + \Wk(\omega) + (-1)^k \Wk(\omega) + (-1)^k \Zk(\omega) )\\
&=\frac{1 + (-1)^k}{4} (\Zk(\omega) + \Wk(\omega) ) , 
\shortintertext{and similarly}
\Projeven \Wk (\omega) &= \frac{1 + (-1)^k}{4} (\Zk(\omega) + \Wk(\omega) ) .
\end{align}

In order to compute the inner products, 
we set $f = \Zk$ and 
\begin{equation}
	\omega = \omega_0 \coloneqq \frac{1}{\sqrt{2}} 
	\begin{pmatrix}
	1 \\
	0 \\
	1 \\
	0
	\end{pmatrix}
	\in \S^3 \cap \Vst[1][-1]
\end{equation}
 in \eqref{eq:Projection Wk}. 
Then we get
\begin{equation}
\ProjVst[1][-1] \ProjVst[-1][1] \Zk(\omega_0) = \normalizeN[2k][0][0] \innerproduct{\Zk}{\Wk}_{L^2(\S^3)} 
\end{equation}
since $\LegendreP(1) = 1$.
The left-hand side can be calculated via the formulas \eqref{eq:ProjV changing V}, \eqref{eq:ProjV polar}, and \eqref{eq:ProjVperp polar}.
For simplicity, we write
\begin{align}
\Zktilde(\omega) 
&\coloneqq \Zk(\orthogonalEst[1][-1] \omega) \\
&= \normalizeN[2k][0][0] \LegendreP( \cos{2\psi} \cos^2{\theta} - \cos{2\varphi} \sin^2{\theta} ) .
\label{eq:Zktilde}
\end{align}
Then we have
\begin{equation}
\omega_0 = \orthogonalEst[1][-1] \omega(0, 0, 0) ,
\end{equation}
so that
\begin{align}
	\ProjVst[1][-1] \ProjVst[-1][1] \Zk(\omega_0)
	\underset{\eqref{eq:ProjV changing V}}&= \ProjV[\planeV_0] \ProjV[\planeV_0^\perp]
	\Zktilde(0, 0, 0) \\
	\underset{\eqref{eq:ProjV polar}, \eqref{eq:ProjVperp polar}}&{=} \frac{1}{4\pi^2} \int_{\psi \in \intervalco{0}{2\pi}} \int_{\varphi \in \intervalco{0}{2\pi}}
	\Zktilde(\varphi, \psi, 0) \, d\varphi \, d\psi \\
	\underset{\eqref{eq:Zktilde}}&= \frac{\normalizeN[2k][0][0]}{2\pi} \int_{\psi \in \intervalco{0}{2\pi}} \LegendreP ( \cos{2\psi} ) \, d\psi .
\end{align}
This integral vanishes when $k$ is odd by \eqref{eq:parity identity}. For even $k$, \cite[7.221.3]{GR2014} gives
\begin{equation} \label{eq:RHS}
\frac{1}{2\pi} \int_{\psi \in \intervalco{0}{2\pi}} \LegendreP[k] ( \cos{2\psi} ) \, d\psi = \lambdak[k/2] .
\end{equation}
This proves \eqref{eq:Zk Wk inner product 2k+1} and \eqref{eq:Zk Wk inner product 2k}.

Now we show the bounds \eqref{eq:Projmean ProjHk upperbound 4k+2 even}, \eqref{eq:Projmean ProjHk upperbound 4k+2 evenperp}, \eqref{eq:Projmean ProjHk upperbound 4k even}, and \eqref{eq:Projmean ProjHk upperbound 4k evenperp}.
\begin{proof}[Proof of \eqref{eq:Projmean ProjHk upperbound 4k+2 even}]
The identity \eqref{eq:Zk Wk Projeven 2k+1} shows that
\begin{align}
	\innerproduct{f}{\Zk[2k+1]}_{L^2(\S^3)}
	= \innerproduct{f}{\Projeven\Zk[2k+1]}_{L^2(\S^3)}
	= 0 , \\
	\innerproduct{f}{\Wk[2k+1]}_{L^2(\S^3)}
= \innerproduct{f}{\Projeven \Wk[2k+1]}_{L^2(\S^3)}
= 0 
\end{align}
for every $f \in \Hkeven[4k+2]$.
Thus, we have 
\begin{equation}
	\ProjV[1, 1] \ProjV[-1, -1] = \ProjV[1, -1] \ProjV[-1, 1] = 0
\end{equation}
on $\Hkeven[4k+2]$ by \eqref{eq:Projection Zk} and \eqref{eq:Projection Wk}, so that \ref{item:functional calculus bound 2} gives the bound $1/2$.
\end{proof}
\begin{proof}[Proof of \eqref{eq:Projmean ProjHk upperbound 4k+2 evenperp}]
The identity \eqref{eq:Zk Wk inner product 2k+1} shows that
\begin{equation}
	\innerproduct{\ProjV[1, 1] \ProjV[-1, -1] f}{\ProjV[1, -1] \ProjV[-1, 1] f} = 0
\end{equation}
for every $f\in\Hk[4k+2]$, so that \ref{item:functional calculus bound 3} gives the bound $3/4$.
\end{proof}
\begin{proof}[Proof of \eqref{eq:Projmean ProjHk upperbound 4k even}]
The identities \eqref{eq:Projection Zk}, \eqref{eq:Projection Wk}, and \eqref{eq:Zk Wk inner product 2k} show that
\begin{equation}\label{eq:pair product norm}
	\norm{\ProjV[1, 1] \ProjV[-1, -1]  \ProjV[1, -1] \ProjV[-1, 1]}_{\Hk[4k] \to \Hk[4k]}
	\underset{\eqref{eq:Projection Zk}, \eqref{eq:Projection Wk}}{=} 
	 \innerproduct{\Zk[2k]}{\Wk[2k]}_{L^2(\S^3)} 
	\underset{\eqref{eq:Zk Wk inner product 2k}}{=} \lambdak ,
\end{equation}
so that \ref{item:functional calculus bound 3+} gives the bound $(3 + \lambdak)/4$.
\end{proof}
\begin{proof}[Proof of \eqref{eq:Projmean ProjHk upperbound 4k evenperp}]
The identity \eqref{eq:Zk Wk Projeven 2k} shows that
\begin{equation}
\innerproduct{f}{\Wk[2k]} \underset{\eqref{eq:Zk Wk Projeven 2k}}= 2 \innerproduct{f}{ \Projeven \Zk[2k] } - \innerproduct{f}{ \Zk[2k] } = - \innerproduct{f}{ \Zk[2k] }
\end{equation}
for every $f \in \Hkeven[4k]^\perp$.
Thus, we have 
\begin{equation}
\Re \innerproduct{\ProjV[1, 1] \ProjV[-1, -1] f}{\ProjV[1, -1] \ProjV[-1, 1] f}
	= - \lambdak \Abs{\innerproduct{f}{\Zk[2k]}}^2
	\leq 0 
\end{equation}
on $\Hkeven[4k]^\perp$ by \eqref{eq:Projection Zk} and \eqref{eq:Projection Wk}, so that \ref{item:functional calculus bound 3} gives the bound $3/4$.
\end{proof}
Finally, we prove \eqref{eq:481 norm upper bound}.
Observe that the spherical harmonic decomposition \eqref{eq:spherical harmonic decomposition on R^d} and the commutation properties \eqref{eq:commute ProjV ProjHk}, \eqref{eq:commute Projeven Projmean} imply
\begin{align} 
	&\quad \norm{\Projmean(1 - \Projrad)}_{L^2(\R^4)\to L^2(\R^4)} \\
	\underset{\eqref{eq:spherical harmonic decomposition on R^d}, \eqref{eq:commute ProjV ProjHk}}&= \sup_{k \geq 1} { \norm{\Projmean}_{\Hk(\S^{3}) \to \Hk(\S^3)} } \\
	\underset{\eqref{eq:commute Projeven Projmean}}&= \sup_{k \geq 1} \begin{aligned}[t]
		\max \{&
		\norm{\Projmean}_{\Hk[2k-1] \to \Hk[2k-1]}, \\
		&\norm{\Projmean}_{\Hkeven[2k] \to \Hkeven[2k]}, \,
		\norm{\Projmean}_{\Hkeven[2k]^\perp \to \Hkeven[2k]^\perp}
		\} .
		\end{aligned}
	\label{eq:481 norm harmonic decomposition}
\end{align}
Therefore, the bounds above give us
\begin{equation}
\norm{\Projmean(1 - \Projrad)}_{L^2(\R^4) \to L^2(\R^4)} \leq \frac{1}{4} (3 + \sup_{k \geq 1} \lambdak ) .
\end{equation}
Here, the sequence $\{ \lambdak \}_{k \in \N}$ is decreasing because
\begin{equation}
\frac{ \lambdak[k+1] }{ \lambdak } = \mleft( \frac{2k+1}{2k+2} \mright)^2 < 1 , 
\end{equation}
so that 
\begin{equation}
\sup_{k \geq 1} \lambdak = \lambdak[1] = \frac{1}{4} .
\end{equation}
This completes the proof.
\end{proof}

\section{Proof of Conjecture \ref{conj:481} with the optimal constant}
\label{section:optimal constant}
Our aim in this section is to give the optimal constant as follows. 
\begin{theorem} \label{thm:exact value}
We have
\begin{equation}
\norm{\Projmean (1-\Projrad)}_{\Leven(\R^4) \to \Leven(\R^4)} = \norm{\Projmean (1-\Projrad)}_{L^2(\R^4) \to L^2(\R^4)} = \frac{15 + \sqrt{97}}{32} .
\end{equation}
As a consequence, Conjecture \ref{conj:481} holds with the constant
\begin{equation} 
	\beta = \frac{17 - \sqrt{97}}{32} ,
\end{equation}
which is the largest possible.
\end{theorem}
In order to establish Theorem \ref{thm:exact value}, it suffices to prove the following results.
\begin{proposition} \label{prop:norm k=4}
	We have
	\begin{equation} \label{eq:norm k=4}
		\norm{\Projmean}_{\Hkeven[4](\S^3) \to \Hkeven[4](\S^3)} 
		= \frac{15 + \sqrt{97}}{32} > \frac{3}{4}.
	\end{equation}
	Moreover, the subspace of $\Hkeven[4]$ consisting of all functions satisfying
	\begin{equation}
	\norm{ \Projmean f }_{L^2(\S^3)} = \frac{15 + \sqrt{97}}{32} \norm{ f }_{L^2(\S^3)}
	\end{equation}
	coincides with the one-dimensional space spanned by
	\begin{align}
	f(\omega) 
	&= 9 \sqrt{2} \Ykmnab[4][0][0][0][0](\omega) - (4 + \sqrt{97}) ( \Ykmnab[4][4][0][0][0](\omega) + \Ykmnab[4][0][4][0][0](\omega) ) .
	\end{align}
\end{proposition}
\begin{proposition} \label{prop:norm k=8}
	We have
	\begin{equation} \label{eq:norm k=8}
		\norm{\Projmean}_{\Hkeven[8](\S^3) \to \Hkeven[8](\S^3)} \leq 3/4.
	\end{equation}
\end{proposition}
We first explain how to prove Theorem \ref{thm:exact value} using the bounds given in Theorem \ref{thm:Projmean ProjHk upperbound} and Propositions \ref{prop:norm k=4}, \ref{prop:norm k=8}.
\begin{proof}[Proof of Theorem \ref{thm:exact value}]
	Observe that the bound \eqref{eq:Projmean ProjHk upperbound 4k even} and the monotonicity of $\{\lambdak\}$ imply
	\begin{equation} \label{eq:Projmean ProjHk upperbound 4k even k>=3}
		\norm{\Projmean}_{\Hkeven[4k] \to \Hkeven[4k]}
		\leq \frac{1}{4} ( 3 + \lambdak[3] )
		= \frac{793}{1024}
		< \frac{15 + \sqrt{97}}{32} 
	\end{equation}
	for $k \geq 3$.
	Combining this with Propositions \ref{prop:norm k=4} and \ref{prop:norm k=8}, we obtain
	\begin{equation}
	\sup_{k \geq 1} {\norm{\Projmean}_{\Hkeven[4k] \to \Hkeven[4k]}} = \frac{15 + \sqrt{97}}{32} .
	\end{equation}
	Thus, the bounds in Theorem \ref{thm:Projmean ProjHk upperbound} and the spherical harmonic decomposition \eqref{eq:481 norm harmonic decomposition} give
	\begin{equation}
	\norm{\Projmean (1-\Projrad)}_{L^2(\R^4) \to L^2(\R^4)} = \frac{15 + \sqrt{97}}{32} .
	\end{equation}
	Moreover, using
	\begin{equation} 
		 \norm{\Projmean(1 - \Projrad)}_{\Leven(\R^4)\to \Leven(\R^4)} 
		= \sup_{k \geq 1} \norm{\Projmean}_{\Hkeven[2k] \to \Hkeven[2k]}
		\label{eq:481 norm harmonic decomposition even}
	\end{equation}
	instead of \eqref{eq:481 norm harmonic decomposition}, we obtain
	\begin{equation*}
	\norm{\Projmean (1-\Projrad)}_{\Leven(\R^4) \to \Leven(\R^4)} = \frac{15 + \sqrt{97}}{32} . \qedhere
	\end{equation*}
\end{proof}

First, we construct an orthonormal basis of $\Hkeven[2k]$ using Proposition \ref{prop:Koornwinder}.
\begin{proposition}\label{prop:basis}
	For each $k \in \N$, the functions
	\begin{equation}\label{eq:Hkeven basis}
		\set{ \Ykmnab[2k][2m][2n][0][0] }{ (m, n) \in \N^2, \, m + n \leq k}
	\end{equation}
	form an orthonormal basis of $\Hkeven[2k](\S^3)$. 
\end{proposition}
\begin{proof}[Proof of Proposition \ref{prop:basis}]
Since the functions \eqref{eq:Hkeven basis} form an orthonormal system in $\Hk[2k]$ due to Proposition \ref{prop:Koornwinder}, it suffices to show that $\Hkeven[2k] = \Projeven \Hk[2k]$ is spanned by them.
For each basis function $\Ykmnab[2k]$ of $\Hk[2k]$, we have
\begin{align}
&\quad \Projeven \Ykmnab[2k](\omega) \\
\underset{\eqref{eq:Projeven polar}}&{=} 
\frac{1}{16} \sum_{ \varepsilon \in \{ -1, 1 \}^2 } \sum_{ j \in \{0, 1\}^2 } \Ykmnab[2k]( \varepsilon_1 \varphi + j_1 \pi , \varepsilon_2 \psi + j_2 \pi , \theta ) \\
\underset{\eqref{eq:Koornwinder functions}}&= 
\frac{1}{16} \normalizeN[2k] (\sin\theta)^m (\cos\theta)^n \JacobiP[k-(m+n)/2][m][n](\cos{2\theta}) \\
&\quad \times \sum_{ \varepsilon \in \{ -1, 1 \}^2 } \sum_{ j \in \{0, 1\}^2 } \cos(m (\varepsilon_1 \varphi + j_1 \pi) - a \pi / 2 ) \cos(n (\varepsilon_2 \psi + j_2 \pi) - b \pi / 2 ) .
\end{align}
Now observe that the sum above vanishes if at least one of $a$ and $b$ equals $1$, or at least one of $m$ and $n$ is odd.
Indeed, 
\begin{itemize}
	\item the sum with respect to $\varepsilon_1$ vanishes when $a = 1$,  
	\item the sum with respect to $\varepsilon_2$ vanishes when $b = 1$, 
	\item the sum with respect to $j_1$ vanishes when $m$ is odd,  
	\item the sum with respect to $j_2$ vanishes when $n$ is odd.
\end{itemize}
Hence, $\Hkeven[2k]$ is spanned by the functions \eqref{eq:Hkeven basis}.
\end{proof}
Now we can prove Propositions \ref{prop:norm k=4} and \ref{prop:norm k=8} by calculating the representation matrices of $\Projmean$ on $\Hkeven[4]$ and $\Hkeven[8]$ with respect to the basis \eqref{eq:Hkeven basis} and finding their maximal eigenvalues.
Here, the matrices are $6 \times 6$ on $\Hkeven[4]$ and $15 \times 15$ on $\Hkeven[8]$. 
Nevertheless, we can reduce to $3 \times 3$ and $6 \times 6$ using the following result.
\begin{proposition}\label{prop:vanishing}
Let $(k, m, n) \in \N^3$ be such that $m+n \leq 2k$.
If at least one of $m$ and $n$ is odd, then we have
	\begin{equation} \label{eq:finite range span} \noeqref{eq:finite range span}
		\Projmean \Ykmnab[4k][2m][2n][0][0] = 0 .
	\end{equation}
\end{proposition}
\begin{proof}[Proof of Proposition \ref{prop:vanishing}]
	First, we recall that
	\begin{equation}
	\Projmean \Ykmnab[4k][2m][2n][0][0] = \Projeven \ProjVst[1][1] \Ykmnab[4k][2m][2n][0][0] 
	\end{equation}
	holds due to \eqref{eq:commute Projeven Projmean}. 
Therefore, it suffices to prove that $\ProjVst[1][1] \Ykmnab[4k][2m][2n][0][0]$ vanishes if at least one of $m$ and $n$ is odd.

For simplicity, we write
\begin{equation}
	\Ykmnabtilde(\omega)
	\coloneqq \Ykmnab[4k][2m][2n][0][0](\orthogonalEst[1][1] \omega) ,
\end{equation}
so that 
\begin{equation}
\ProjVst[1][1] \Ykmnab[4k][2m][2n][0][0] ( \orthogonalEst[1][1] \omega ) = \ProjV[\planeV_0] \Ykmnabtilde ( \omega ) .
\end{equation}
Then the formula \eqref{eq:ProjV polar} gives us
\begin{equation}
\ProjV[\planeV_0] \Ykmnabtilde 
= \begin{dcases}
\frac{1}{2\pi} \int_{\varphi \in \intervalco{0}{2\pi}} \Ykmnabtilde(\phantom{\pi -{}}\varphi, \psi, \theta) \, d\varphi , \\
\frac{1}{2\pi} \int_{\varphi \in \intervalco{0}{2\pi}} \Ykmnabtilde(\phantom{\pi} - \varphi, \psi, \theta) \, d\varphi , \\
\frac{1}{2\pi} \int_{\varphi \in \intervalco{0}{2\pi}} \Ykmnabtilde(\pi - \varphi, \psi, \theta) \, d\varphi 
\end{dcases}
\end{equation}
by changing the variable of integration.
On the other hand, we have
\begin{align}
\orthogonalEst[1][1]\omega(\phantom{\pi} - \varphi,\psi,\theta) &= 
\begin{pmatrix}
	1 & 0 & 0 & 0 \\
	0 & 0 & 0  & 1 \\
	0 & 0 & 1 & 0 \\
	0 & 1 & 0 &  0
\end{pmatrix} \orthogonalEst[1][1]\omega(\varphi,\psi,\theta) , \\
\orthogonalEst[1][1]\omega(\pi - \varphi,\psi,\theta) &= 
\begin{pmatrix}
	0 & 0 & 1 & 0 \\
	0 & 1 & 0 & 0 \\
	1 & 0 &0 & 0 \\
	0 & 0 & 0 & 1  
\end{pmatrix} \orthogonalEst[1][1]\omega(\varphi,\psi,\theta) 
\end{align}
due to 
\begin{equation}
	\orthogonalEst[1][1]\omega(\varphi,\psi,\theta)
	=\frac{1}{\sqrt{2}}
	\begin{pmatrix}
		\cos\psi \cos\theta + \cos\varphi \sin\theta\\
		\sin\psi \cos\theta + \sin\varphi \sin\theta\\
		\cos\psi \cos\theta - \cos\varphi \sin\theta\\
		\sin\psi \cos\theta - \sin\varphi \sin\theta
	\end{pmatrix} .
\end{equation}
In other words, replacing $\varphi$ by $- \varphi$ in the polar coordinates interchanges the second and fourth components in the Cartesian coordinates, 
and replacing $\varphi$ by $\pi - \varphi$ interchanges the first and third.
Substituting these into the definition of $\Ykmnab[4k][2m][2n][0][0]$, we obtain
\begin{equation}
\Ykmnabtilde(\varphi, \psi, \theta) = (-1)^m \Ykmnabtilde( - \varphi, \psi, \theta) = (-1)^n \Ykmnabtilde(\pi - \varphi, \psi, \theta) .
\end{equation}
Hence, integrating this with respect to $\varphi \in \intervalco{0}{2\pi}$ yields
\begin{equation}
\ProjV[\planeV_0] \Ykmnabtilde = (-1)^m \ProjV[\planeV_0] \Ykmnabtilde = (-1)^n \ProjV[\planeV_0] \Ykmnabtilde .
\end{equation}
Therefore, $\ProjV[\planeV_0] \Ykmnabtilde$ vanishes if at least one of $m$ and $n$ is odd.
\end{proof}
Now we prove Propositions \ref{prop:norm k=4} and \ref{prop:norm k=8}.
\begin{proof}[Proof of Proposition \ref{prop:norm k=4}]
	Thanks to Proposition \ref{prop:vanishing}, it suffices to find the operator norm on the subspace spanned by the ordered orthonormal basis
	\begin{equation}
	 (\Ykmnab[4][0][0][0][0], \Ykmnab[4][4][0][0][0], \Ykmnab[4][0][4][0][0] )
	\end{equation}
because $\Projmean$ is self-adjoint on $\Hkeven[4]$. 
One can find that the representation matrix of $\Projmean$ on this subspace is
	\begin{equation}\label{eq:representation matrix k=4}
		\frac{1}{2^6}
		\begin{pmatrix}
			22 & -9\sqrt{2} & -9 \sqrt{2} \\
    -9\sqrt{2} &         35 &          3  \\
	-9\sqrt{2} &          3 &         35
		\end{pmatrix} , 
	\end{equation}
and its characteristic polynomial is
	\begin{equation}\label{eq:characteristic polynomial k=4}
	\frac{1}{2^5} \mleft( 2 \lambda - 1 \mright)
		\mleft( 16 \lambda^2 - 15 \lambda + 2 \mright).
	\end{equation}
As a consequence, the maximal eigenvalue of \eqref{eq:representation matrix k=4} is $(15+\sqrt{97})/32$, and its associated eigenspace is spanned by
\begin{equation}
\begin{pmatrix}
9 \sqrt{2} \\
- (4 + \sqrt{97}) \\
- (4 + \sqrt{97})  
\end{pmatrix} .
\end{equation}
This proves Proposition \ref{prop:norm k=4}.
\end{proof}

\begin{proof}[Proof of Proposition \ref{prop:norm k=8}]
It suffices to find the operator norm on the subspace spanned by the ordered orthonormal basis
	\begin{equation}
	( \Ykmnab[8][0][0][0][0], \Ykmnab[8][4][0][0][0], \Ykmnab[8][0][4][0][0],
	 \Ykmnab[8][8][0][0][0], \Ykmnab[8][4][4][0][0], \Ykmnab[8][0][8][0][0] ) .
	\end{equation}
One can find that the representation matrix of $\Projmean$ on this subspace is
	\begin{equation} \label{eq:representation matrix k=8}
		\setlength{\arraycolsep}{3.5pt}
\frac{1}{2^{14}}
\left(\begin{array}{r@{}l r@{}l r@{}l r@{}l r@{}l r@{}l}
	3574&&-1660&\sqrt{2}&-1660&\sqrt{2}&1225&\sqrt{2}&18&\sqrt{70}&1225&\sqrt{2}\\
	-1660&\sqrt{2}&5040&&-1360&&-2772&&-168&\sqrt{35}&-980&\\
	-1660&\sqrt{2}&-1360&&5040&&-980&&-168&\sqrt{35}&-2772&\\
	1225&\sqrt{2}&-2772&&-980&&6435&&198&\sqrt{35}&35&\\
	18&\sqrt{70}&-168&\sqrt{35}&-168&\sqrt{35}&198&\sqrt{35}&4900&&198&\sqrt{35}\\
	1225&\sqrt{2}&-980&&-2772&&35&&198&\sqrt{35}&6435&
\end{array}\right) , 
	\end{equation}
and its characteristic polynomial is
	\begin{equation}\label{eq:characteristic polynomial k=8}
	\frac{1}{2^{16}} \lambda ( 2 \lambda - 1 ) ( 32 \lambda - 9 )
		(1024 \lambda^3 - 1164 \lambda^2 + 349 \lambda - 29) .
	\end{equation}
	Here, the cubic factor can be rewritten as
	\begin{align}\label{eq:characteristic polynomial k=8 rewritten}
		&\quad 1024 \lambda^3 - 1164 \lambda^2 + 349 \lambda - 29 \\
		&=1024 ( \lambda - 3/4 )^3
		+1140 ( \lambda - 3/4 )^2
		+331 ( \lambda - 3/4 ) 
		+ 10 , 
	\end{align}
	so that it does not have any zeros on $\intervalco{3/4}{\infty}$.
	As a consequence, the maximal eigenvalue of \eqref{eq:representation matrix k=8} is less than $3/4$. 
	This proves Proposition \ref{prop:norm k=8}.
\end{proof}

\section*{Use of Artificial Intelligence}
During the preparation of this work the author used GPT-6 Astra (OpenAI) to assist with the calculations in the proofs of Propositions \ref{prop:norm k=4} and \ref{prop:norm k=8}.
After using this tool, the author reviewed and edited the content as needed and takes full responsibility for the content of the manuscript.

\end{document}